\documentclass[10pt]{article}%
\usepackage{float} 
\usepackage{fullpage}
\usepackage{amsfonts}
\usepackage{amsmath}
\usepackage{amssymb}
\usepackage{amsbsy}
\usepackage{amsthm}
\usepackage{mathtools}
\usepackage{epsfig}
\usepackage{graphicx}
\usepackage{colordvi}
\usepackage{graphics}
\usepackage{xcolor}
\usepackage{bm}
\usepackage{verbatim} 
\usepackage{cite}
\usepackage[normalem]{ulem}

\numberwithin{equation}{section}

\usepackage{hyperref}
\usepackage{comment}

\usepackage[numbers,sort]{natbib}

\newtheorem{theorem}{Theorem}[section]
\newtheorem{thm}{Theorem}[section]

\newtheorem{lemma}{Lemma}[section]

\newtheorem{algorithm}[thm]{Algorithm}

\newcommand{\norm}[1]{\ensuremath{\left\|{#1}\right\|}}
\newcommand{\inv}{^{-1}}
\newcommand{\ohm}{ \Omega }
\newcommand{\Grad}{\ensuremath{\nabla}}
\def\al#1\eal{\begin{align}#1\end{align}}
\def\als#1\eals{\begin{align*}#1\end{align*}}
\newcommand{\pare}[1]{\left({}#1\right)}

\newcommand{\tv}{\ensuremath{\tilde{v}}}
\newcommand{\te}{\ensuremath{\tilde{e}}}

\def\div{\nabla \cdot}

\begin{document}

\title{Continuous data assimilation in steady Navier-Stokes equations with unknown viscosity: robust and efficient solvers and fast parameter recovery}

\author{
Leo G. Rebholz\thanks{\small School of Mathematical and Statistical Sciences, Clemson University, Clemson, SC, 29364.  Partially supported by Department of Energy grant DE-SC0025292  (rebholz@clemson.edu)}
\and
Jorge Reyes\thanks{\small Department of Mathematics, Texas State, San Marcos, TX 78666, USA. (jreyes@txstate.edu)}
\and
Jared P. Whitehead\thanks{\small Department of Mathematics, Brigham Young University, Provo, Utah 84602, USA.   Partially supported by NSF grants DMS-2510495 and CCF-2343286 (whitehead@mathematics.byu.edu)}
}
\maketitle

\begin{abstract}
Recent advances in equation discovery methods such as SINDy
have highlighted the growing interest in identifying governing parameters and models directly from data. In this work, we take a complementary approach grounded in analysis and numerical PDE methods: we recover an unknown viscosity in steady Navier-Stokes equations (NSE) from partial incompressible flow observations using continuous data assimilation (CDA).  We propose a simple and efficient parameter recovery algorithm and also a nonlinear solver for CDA-NSE.  Together, this creates a highly efficient technique for recovering an unknown viscosity from partial solution data.  Our analysis establishes the well-posedness of steady CDA-NSE, quadratic convergence of the parameter recovery algorithm, and quadratic convergence of a CDA-Picard + CDA-Newton nonlinear solver.  Numerical experiments illustrate that the methods are very effective in restoring parameters quickly, even with poor initial guesses.
\end{abstract}

\section{Introduction}

In this paper, we consider efficient parameter recovery algorithms for the steady Navier-Stokes equations (NSE)  that use known partial solution data together with continuous data assimilation (CDA) to recover an unknown viscosity.  The steady NSE take the following form, for velocity $u$ and pressure $p$, on a domain $\Omega\subset\mathbb{R}^d$, d=2 or 3:
\begin{align}
- \nu \Delta u + u\cdot\nabla u + \nabla p  & = f,  \label{nse1} \\
\nabla \cdot u&=0, \label{nse2}
\end{align}
together with appropriate boundary conditions, where $\nu$ is the kinematic viscosity, $f$ is an external forcing such as gravity or bouyancy.  This system models many types of flows across the spectrum of science and engineering \cite{Layton08,G94,John16}, and thus recovering unknown viscosity parameters from partial solution data or observations is of widespread interest.

We consider the setting where $\nu$ is unknown.  However, partial solution data is known, meaning we can access $I_H u$ where $I_H$ is an interpolant with characteristic point spacing $H$.  This leads to the following CDA-NSE system, where $v$ and $q$ are unknowns, $\hat\nu$ is a known approximation (perhaps a guess) of $\nu$:
\begin{align}
-\hat \nu \Delta v + v\cdot\nabla v + \nabla q  + \mu I_H (v-u)& = f,  \label{cdanse1} \\
\nabla \cdot v&=0, \label{cdanse2}
\end{align}
where $v$ satisfies the same boundary conditions as $u$.  Note that if $\hat\nu=\nu$, then solving \eqref{cdanse1}-\eqref{cdanse2} recovers $u$.  For simplicity of analysis, we will consider enclosed flows with no-slip boundary conditions, i.e.
\[
u |_{\partial\Omega} = v |_{\partial\Omega} =0,
\]
but our analysis is extendable to nonhomogeneous mixed Dirichlet-Neumann boundary conditions.

Herein we propose two new ideas that combine to create a highly efficient method that recovers the true viscosity $\nu$ from partial solution data $I_H u$.  The first idea is based on our analysis that takes advantage of CDA to show 1) the quantity
\[
E(\hat\nu) = \frac12 \| I_H (u-v(\hat\nu) )\|^2 \sim |\nu-\hat\nu|^2, 
\]
which implies that the $E(\hat\nu)$ is quadratic with vertex at $\hat\nu=\nu$; and 2) $E'(\hat\nu)$ can be easily calculated.  Hence a (modified) Newton's method can be constructed to find the root $\nu$ of $E$.  This algorithm takes the form
\[
\nu_{k+1} = \nu_k - \frac{2 E(\nu_k)}{E'(\nu_k)},
\]
noting the modification with the 2 is used because we show $E$ is parabolic and $E'(\nu)=0$.  Calculating $E(\nu_k)$ and $E'(\nu_k)$ requires solving \eqref{cdanse1}-\eqref{cdanse2}, and the second novel idea of this paper is to efficiently solve this system by extending the Picard+Newton iteration methodology of \cite{PRTX25} to CDA-Picard+CDA-Newton.  Picard+Newton is proven in \cite{PRTX25} to be both unconditionally stable and quadratically convergent, and very robust with respect to the viscosity (far more robust than either Picard or Newton alone).  CDA is incorporated into Picard and Newton following the recent work in \cite{LHRV23,GLNR24}, and we prove herein that CDA-Picard+CDA-Newton maintains both unconditional stability and quadratic convergence but is even faster and more robust with respect to the viscosity and initial guess due to its incorporation of data.  Hence the combination of these two key new ideas yields a method that can solve  \eqref{cdanse1}-\eqref{cdanse2} quickly, and with just a few solves of  \eqref{cdanse1}-\eqref{cdanse2} we can recover the unknown viscosity $\nu$.

To our knowledge, this is the first work where CDA is used to recover unknown parameters in a steady system.  CDA was originally developed in 2014 to help time dependent simulations remain accurate over longer times \cite{AOT_2014,AT14} and more precisely, for many dissipative systems it can be proven that with enough partial solution data the computed solution converges to the true solution exponentially fast in time, see e.g. \cite{CDA_5_2015,LP24,LRZ19,ZRSI19,GN20}.  Since these original papers of Azouani et al, CDA has caught fire in the applied mathematics and scientific computing communities due its strong mathematical foundation and that it allows for rigorous proofs of accuracy.  It has been used on a wide variety of dissipative PDEs with physical applications e.g. \cite{FLT16,CDA_5_2015,FLT19,Biswas_Hudson_Larios_Pei_2017,LV24,CL21}, and many improvements have been made for its implementation e.g. \cite{HRV24,RZ21,CFLS25}.  In particular, an efficient and effective approach to the data assimilation parameter recovery problem based on CDA was recently discovered by Carlson, Hudson and Larios in \cite{CHL_2020} for the time dependent NSE.  They showed one can exploit scalings of parameters with the interpolant of the error (which is known due to known partial solution data) to quickly recover the Reynolds number on the fly in time dependent incompressible flow simulations using the NSE, with negligible extra cost.  This breakthrough idea was further developed and generalized in \cite{M22,Lorenz_2021,FLMW24,newey2025model,ABH26} to different systems and recovering various parameters or an unknown forcing, and in \cite{BB26,L23} to sharpen regularity and uniqueness properties of systems.

The first use of CDA for steady problems was the 2023 paper \cite{LHRV23}, where the connection between nudging in time stepping and nudging at nonlinear iterations was exploited to show that CDA can accelerate and even enable convergence of a nonlinear solver.  They showed that with enough partial solution data, the Picard solver for the Navier-Stokes equations converges linearly to the true solution, even for arbitrarily bad initial guesses and arbitrarily large Reynolds numbers.  Moreover, the more partial solution data one has the faster the convergence, and in the setting where multiple NSE solutions exists the methods will converge to the solution from which the partial solution data was obtained.  For Newton, the paper showed that the convergence basin radius grows as more partial solution data is known.  CDA has since been studied for nonlinear solvers of the steady Boussinesq equations \cite{H25}, with noisy data \cite{GLNR24}, using efficient splitting method implementations \cite{FRV26,FFR26}, and for the p-Laplace equation \cite{M26}.  Herein, we tackle a natural next step of CDA for steady problems by developing an efficient method for using CDA with partial solution data to recover an unknown viscosity for the NSE.  While we focus on the steady NSE, the ideas we present are likely adaptable to a wide range of steady problems.

The paper is arranged as follows: In Section \ref{sec: Prelim} we go over notation and necessary mathematical preliminaries. Section \ref{sec: main} introduces the CDA-NSE system under inconsistent viscosity, and proves mathematical properties for it including consistency error and well-posedness.  In Section \ref{sec: Analysis}, we propose the CDA-Picard+CDA-Newton solver for the CDA-NSE system under inconsistent viscosity, and prove that provided enough partial solution data the new solver is stable and quadratically  convergent.  In Section \ref{sec:paramrec}, we derive the modified Newton viscosity recovery algorithm.  Section \ref{sec: Numres} gives numerical results for several benchmark tests that show the efficiency and effectiveness of both the CDA-Picard+CDA-Newton solver and the modified Newton viscosity recovery algorithm.
We conclude our work in Section \ref{sec: Con} and offer multiple directions for future work.

\section{Notation and Preliminaries} \label{sec: Prelim}
 The solutions are sought in the following functional spaces:
\begin{align*}
    X &:= (H_0^1(\Omega))^d= \left\lbrace v \in L^2(\Omega)^d,\  \nabla v \in L^2(\Omega)^{d\times d}:  v|_{\partial\Omega} = 0 \; \right\rbrace\\
Q &:= L_0^2(\Omega) = \left\lbrace p \in L^2(\Omega) \ : \ \int_\Omega \! p \ \mathrm{d}x  = 0 \right\rbrace,
\end{align*}
where the domain $\Omega$ is an open, connected, bounded set with either a $C^2$ smooth boundary, or is a convex polygon/polyhedron.  

The $L^2(\Omega)$ norm and inner product are denoted by $\norm{\cdot }$ and $( \cdot, \cdot)$ respectively. Furthermore, the dual norm of $ X$, $ H\inv(\Omega)$, will be denoted by $ \norm{\cdot}_{-1}$ and the notation $\langle \cdot , \cdot \rangle$ is used for the dual pairing of $X$ and $X'$.  Norms in the Sobolov spaces $H^k{(\Omega)}$ are written $\norm{ \cdot }_k$, and, likewise, semi-norms are given by $| \cdot |_k$.

The space of weakly divergence-free functions in $X$ are given by
\begin{equation*}
V = \lbrace v \in X : (q, \nabla \cdot v) = 0, \forall q \in Q \rbrace.
\end{equation*}
Observe that since $\nabla \cdot v \in Q$, we can rewrite $V=\{ v\in X,\ \| \nabla \cdot v\|=0\}.$

We recall that the Poincaré inequality holds on $X$: there exists a constant $C_P$, depending only on the domain, such that $ \forall \phi \in X$,
$
\|\phi\| \leq C_P \|\nabla \phi\|.
$
This shows that the $H^1$ norm is equivalent to the $L^2$ norm of the gradient in our setting, and we will use the latter as the norm for both $X$ and $V$.

For $ u,v,w \in X $, we define the trilinear form $ b$ as follows: 
\begin{eqnarray}
 b(u,v,w) = (u\cdot\nabla v,w). \label{eq: B_term}
 \end{eqnarray}
We note that $ \eqref{eq: B_term}$ as well as other formulations such as skew-symmetric, EMAC or rotational form are all equivalent in the continuous setting, as well as discretely if strongly divergence free elements are used in a finite element discretization (e.g. Scott-Vogelius). However, when using only weakly divergence free elements (e.g. Taylor-Hood, MINI) the different forms exhibit different conservation behavior~\cite{CHOR17}, and $b$ is typically adjusted through a skew-symmetrization.

We also recall the following lemmas, which will be used in our analysis.
\begin{lemma}[\cite{Layton08,temam}]\label{TRIL}
For $u, v, w \in X$, there exists a constant $M$ that depends only on the domain $ \ohm$ such that
\al
 b(u, v, w) &\leq   M \norm{u}^{\frac{1}{2}} \norm{\Grad u}^{\frac{1}{2}} \norm{\Grad v}  \norm{\Grad w},\label{eq: bbound1} \\
 b(u, v, w) &\leq  M \norm{\Grad u} \norm{ \Grad v}  \norm{\Grad w},\label{eq: bbound2} \\
 b(u,v,w) &\leq M \norm{\Grad u} \norm{\Grad v} \norm{w}^{1/2} \norm{\Grad w}^{1/2}.\label{eq: bbound3}
\eal
\end{lemma}

The weak form of the steady state NSE \eqref{nse1}-\eqref{nse2} is given by: find $ \pare{u,p} \in X \times Q$ such that $ \forall \chi \in X$ and $ \forall q \in Q$ 
\al
\nu \pare{ \Grad u, \Grad \chi } + b(u,u,\chi) + (p, \div \chi) &= \langle f,\chi \rangle, \label{eq: weak_NSE1} \\
\pare{\div u, q} &= 0 \label{eq: weak_NSE2}.
\eal
Furthermore, since the pair $(X,Q)$ satisfy the inf-sup condition, we can consider the following equivalent formulation: Find $ u \in V$ such that $ \forall \chi \in V$
\al
\nu \pare{ \Grad u, \Grad \chi } + b(u,u,\chi)  &= \langle f,\chi \rangle. \label{eq: weak_NSE3} 
\eal

\begin{lemma}[\cite{Layton08}]
Let $ \alpha_0 = M \nu^{-2}\norm{f}_{-1}$. For any $ f \in H\inv(\ohm)$ and $ \nu >0$, there exists at least one solution for the NSE \eqref{eq: weak_NSE1}-\eqref{eq: weak_NSE2}.  Moreover, any such solution satisfies
\begin{equation}
    \| \nabla u \| \le \nu^{-1} \| f \|_{-1}. \label{nsestab}    
\end{equation}
If $ \alpha_0 < 1$ then solutions are unique.
\end{lemma}

\subsection{CDA Background}

We assume that the interpolant $I_H$ is linear and that there exists a constant $C_I>0$ such that 
\begin{align}
\| I_H \phi \| & \le C_I \| \phi \|,\label{interp1} \\
\| \phi - I_H \phi \| & \le C_I H \| \nabla \phi \|, \label{interp2}
\end{align}
for all $\phi\in X$.  It is known that the Scott-Zhang interpolant and the $L^2$ projection onto piecewise constants both satisfy these constraints \cite{GN20}.
Under this assumption, we have the following bound for any $\phi \in X$: \cite{GLNR24,GN20}
\begin{equation}
\frac{\nu}{2} \| \nabla \phi \|^2 + \lambda \| \phi \|^2
\le
\nu \| \nabla \phi \|^2 + \mu \| I_H \phi \|^2, \label{lowerboundlam}
\end{equation}
where $\lambda = \min \{ \mu, \frac{\nu}{2C_I^2 H^2} \}.$

The convergence of the CDA-Picard+CDA-Newton algorithm in Theorem \ref{Thm: CDA-Picard_res} is in terms of the following weighted $H_1$ norm.
\begin{equation} \label{eq: starnorm}
    \norm{\phi}_* = \pare{ \frac{\hat\nu}{4} \norm{ \Grad \phi }^2 +  \lambda \norm{ \phi }^2 }^{1/2}.
\end{equation}

\section{CDA-NSE with inconsistent viscosity: well-posedness and consistency error}
\label{sec: main}

This section focuses on the CDA-NSE system and establishes some basic properties for that system that are necessary to justify the proposed algorithms.  The system comes from the steady NSE combined with CDA that nudges towards partial solution data that may not be consistent with the best known approximation of the viscosity.  The CDA-NSE system takes the form: Given $f\in H^{-1}(\Omega)$, $\hat \nu>0$, $\mu>0$ and $I_Hu$, find $v\in V$ satisfying
\begin{align}
\hat \nu(\nabla v,\nabla \chi) + b(v,v,\chi) + \mu(I_H (v-u),I_H \chi)
=\langle f,\chi \rangle. \label{cdanse}
\end{align}
The parameter $\mu>0$ is a nudging parameter, which enforces how strongly the nudging penalizes the difference between the interpolant of the CDA-NSE solution and the interpolant of the true solution.  The analysis below will guide the choices of $H$ and $\mu$.

We first consider stability of the system \eqref{cdanse}.
\begin{lemma}\label{cdansestab}
Solutions to CDA-NSE \eqref{cdanse} satisfy the a priori bound
\begin{equation}
 \| \nabla v \| \le \left( \hat \nu^{-2} + \mu C_P^2 C_I^2 \nu^{-2}\hat \nu^{-1} \right)^{1/2} \| f \|_{-1} := C_v.
\end{equation}


\end{lemma}
\begin{proof}
Choosing $\chi=v$ in \eqref{cdanse} eliminates the nonlinear term and gives the equation
\[
\hat \nu \| \nabla v \|^2 + \mu \| I_H v\|^2 = \langle f,v \rangle  + \mu(I_H u,I_H v).
\]
Using the definition of the $H^{-1}$ norm on the first term and the Cauchy-Schwarz inequality on the second term, we obtain the bound
\[
\hat \nu \| \nabla v \|^2 + \mu \| I_H v\|^2 = \| f\|_{-1} \| \nabla v \| +  \mu \| I_H u \| \| I_H v\|.
\]
Now using Young's inequality on both right hand side terms and then reducing provides us with
\[
\hat \nu \| \nabla v \|^2 + \mu \| I_H v\|^2 = \hat \nu^{-1} \| f\|_{-1}^2  +  \mu \| I_H u \|^2.
\]
Using the $L^2$ stability of $I_H$ from \eqref{interp1} and Poincar\'{e}'s inequality followed by the a priori bound on $\| \nabla u\|$ from \eqref{nsestab}, we upper bound the last right hand side term and obtain
\[
\hat \nu \| \nabla v \|^2 + \mu \| I_H v\|^2 = \hat \nu^{-1} \| f\|_{-1}^2  +  \mu C_P^2 C_I^2 \nu^{-2} \| f \|_{-1}^2.
\]
Reducing completes the proof. 
\end{proof}

We next establish well-posedness of \eqref{cdanse}, provided enough partial solution data is known.
\begin{theorem}\label{cdansewp}
If  $H \leq \frac{\sqrt{2}\hat{\nu}^2}{\sqrt{27} C_I M^2 C_v^2} $ and $\mu\ge \frac{\hat \nu}{2C_I^2 H^2}$, then the CDA-NSE system \eqref{cdanse} is well-posed.
\end{theorem}
\begin{proof}
Since $\mu>0$ is fixed and $\| u \| \le C_P \| \nabla u\| \le C_P \nu^{-1} \| f\|_{-1}$, thanks to \eqref{interp1} and Lemma \ref{cdansestab}, the existence of solutions to \eqref{cdanse} can be established using the Leray-Schauder theorem by mirroring the proof of weak NSE solution existence in \cite{Layton08}.

For uniqueness, suppose that $v_1, v_2 \in V$ are both solutions to the CDA–NSE system. Subtracting the equation \eqref{cdanse} satisfied by $v_2$ from that satisfied by $v_1$ yields

\[
\hat \nu(\nabla e,\nabla \chi) + b(v_1,e,\chi) + b(e,v_2,\chi) + (I_H e,I_H \chi)=0,
\]
where $e=v_1-v_2$.  Taking $\chi=e$ makes the first nonlinear term vanish and yields 
\[
\hat \nu \| \nabla e \|^2 + \mu \| I_H e \|^2 = -b(e,v_2,e).
\]
Lower bounding the left side using \eqref{lowerboundlam}
 and upper bounding the right side using \eqref{eq: bbound1} we obtain the estimate
 \begin{equation}
\frac{ \hat \nu }{2} \| \nabla e \|^2 + \lambda \| e \|^2
\le M \| \nabla v_1 \| \| e \|^{1/2} \| \nabla e \|^{3/2},
\end{equation}
where $\lambda=\min \{ \mu, \frac{\hat \nu}{2 C_I^2 H^2}$ \}.  Next using Lemma \ref{cdansestab} and Young's inequality, we get that 
 \begin{equation}\label{eq: 3.4}
\frac{ \hat \nu }{2} \| \nabla e \|^2 + \lambda \| e \|^2
\le \frac{27}{4} M^4 C_v^4 \hat\nu^{-3} \| e \|^2 + \frac{\hat \nu}{4}\| \nabla e \|^{2}.
\end{equation}
Reducing and using the assumed bound that $\mu\ge \frac{\hat \nu}{2C_I^2 H^2}$ provides us with
\begin{equation}
\frac{ \hat \nu }{4} \| \nabla e \|^2 
+ \left( \frac{\hat\nu}{2 C_I^2 H^2}  - \frac{27}{4} M^4 C_v^4 \hat\nu^{-3} \right) \| e \|^2 \le 0.
\end{equation}
Finally the assumption on the size of $H$ completes the proof of uniqueness.  Hence we have established existence and uniqueness of the system \eqref{cdanse} as well as its continuous dependence on the data in Lemma \ref{cdansestab}.  This completes the proof.
\end{proof}

With the well posedness of CDA-NSE established, we next prove an upper bound for its consistency error arising from the mismatch in viscosity.

\begin{theorem}\label{cdanseconsistency}
Let $u$ be the NSE solution of \eqref{eq: weak_NSE3} from which the partial solution data $I_H u$ is obtained, and let $v$ solve \eqref{cdanse}.  Then under the assumptions of Theorem \ref{cdansewp} and if $H<\left(46C_I \nu^{3/2} \hat \nu^{-3/2} \alpha_0^2 \right) ^{-1}$, 
\begin{align*}
   \frac{\| \nabla (u-v)   \|}{\norm{\Grad u} }  & \le 
2\sqrt 2 \frac{| \hat \nu- \nu|}{ \hat \nu},  \\
\| u-v \| & \le 
 \sqrt{2} C_I H \frac{|\hat \nu- \nu|}{\hat\nu \nu} \norm{f}_{-1}.
\end{align*}

\end{theorem}

\begin{proof}
Subtracting CDA-NSE \eqref{cdanse} from the NSE \eqref{eq: weak_NSE3} gives us the error equation
\[
\hat \nu (\nabla e,\nabla \chi) + (\nu-\hat \nu)(\nabla u,\nabla \chi) + b(v,e,\chi) + b(e,u,\chi) + \mu(I_H e,I_H \chi) = 0 \ \forall \chi \in V,
\]
where we denote $e=u-v$.  Taking $\chi=e$ eliminates the first nonlinear term and provides us with
\[
\hat \nu \| \nabla e \|^2 + \mu \| I_H e \|^2 = (\hat \nu- \nu)(\nabla u,\nabla e) - b(e,u,e).
\]
Lower bounding the left hand side using \eqref{lowerboundlam}, and using Cauchy-Schwarz and \eqref{eq: bbound1} on the right hand side gives
\begin{equation}
\frac{\hat \nu}{2} \| \nabla e \|^2 + \lambda \| e \|^2 \le 
|\hat \nu- \nu| \| \nabla u \|  \| \nabla e \|  + M \| \nabla u \| \| e \|^{1/2} \| \nabla e \|^{3/2}.
\end{equation}
Now using Young's inequality on both right hand side terms and using the a priori NSE stability bound \eqref{nsestab}, we get 
\begin{equation}
\frac{\hat \nu}{8} \| \nabla e \|^2 + \lambda \| e \|^2 \le 
\frac{|\hat \nu- \nu|^2}{\hat\nu} \norm{\Grad u}^2   + 54 \hat \nu^{-3} \nu^{4} \alpha_0^4 \| e \|^{2}. \label{Z1}
\end{equation}

Next using the assumptions on $\mu$ and $H$, $( H < \pare{ 6\sqrt{3} \hat{\nu}^{-2} \nu^2 C_I \alpha_0^2  }\inv) $, we get the following relative error bound,
\begin{equation}
   \frac{\| \nabla (u-v)   \|}{\norm{\Grad u} }  \le 
2\sqrt 2 \frac{| \hat \nu- \nu|}{ \hat \nu}.   
\end{equation}
For the $L^2$ bound, from \eqref{Z1} we use  the stability bound \eqref{nsestab} and the assumption on $H$ to get the following error bound,
$$ \norm{e} \leq \sqrt{2} C_I H \frac{|\hat \nu- \nu|}{\hat\nu \nu} \norm{f}_{-1}. $$

This completes the proof.

\end{proof}

\section{CDA-Picard + CDA-Newton to solve CDA-NSE with inconsistent viscosity}
\label{sec: Analysis}

We now present and analyze our new solver for the NSE when partial solution data and a potentially mismatched viscosity is given. The CDA-Picard + CDA-Newton iteration for solving the CDA-NSE system \eqref{cdanse} is given by: \\~ \\
Step 1 (CDA-Picard): Given $v_k\in V$, first find $\tilde v_{k+1}\in V$ satisfying for all $\chi \in V$,
    \begin{equation}
        \hat\nu (\nabla \tilde v_{k+1} , \nabla \chi) 
        + b(v_k, \tilde v_{k+1}, \chi)
        +\mu (I_H(\tilde v_{k+1} - u),I_H\chi) 
        = 
        \langle f, \chi\rangle. \label{CDAPicard}
    \end{equation}
Step 2 (CDA-Newton using CDA-Picard solution as input): find $v_{k+1}\in V$ satisfying  for all $\chi \in V$,
   \al
        \hat\nu(\Grad v_{k+1} , \Grad \chi) + b(\tilde v_{k+1}, v_{k+1}, \chi) +b(v_{k+1},\tilde v_{k+1}, \chi) - b(\tilde {v}_{k+1},\tilde {v}_{k+1},\chi)\nonumber \\
         +\mu (I_H(v_{k+1} - u), I_H\chi) = \langle f, \chi \rangle.\label{CDANewton}
    \eal


We first prove an a priori bound for the Step 1 (CDA-Picard) solution.
\begin{lemma}[Stability of CDA-Picard]
\label{lem: CDA_P_Stab}
The solution of the Step 1 CDA-Picard iteration satisfies
  \begin{equation}
        \frac{\hat\nu}{2} \norm{\nabla \tv_{k+1} }^2 
        +\lambda \norm{\tv_{k+1}}^2 
        \le \hat\nu^{-1} \norm{ f }_{-1}^2 + \mu \norm{I_H(u)}^2.  \label{CDAPicard0}
    \end{equation}
\end{lemma}

\begin{proof}

    Choose $ \chi = \tv_{k+1}$ in \eqref{CDAPicard} resulting in
    \begin{equation}
        \hat\nu \norm{\nabla \tv_{k+1} }^2 
        +\mu (I_H(\tv_{k+1} - u),I_H \tv_{k+1}) 
        = 
        \langle f, \tv_{k+1} \rangle. \label{CDAPicard2}
    \end{equation}
    Since $ I_H$ is linear we can rearrange into
    \begin{equation}
        \hat\nu \norm{\nabla \tv_{k+1} }^2 
        +\mu \norm{I_H \tv_{k+1}} ^2
        = 
        \langle f, \tv_{k+1} \rangle + \mu (I_H u, I_H \tv_{k+1}).  \label{CDAPicard3}
    \end{equation}
Applying Cauchy-Schwarz and Young inequalities on the right hand side terms, we simplify the estimate to
    \begin{equation}
        \frac{\hat\nu}{2} \norm{\nabla \tv_{k+1} } ^2
        +\frac{\mu}{2} \norm{I_H \tv_{k+1}} ^2
        \leq
        \frac{\hat\nu^{-1}}{2} \norm{ f }_{-1}^2 + \frac{\mu}{2} \norm{I_H(u)}^2.  \label{CDAPicard4}
    \end{equation}
Next, we use \eqref{lowerboundlam} to lower bound the left hand side and obtain
    \begin{equation}
        \frac{\hat\nu}{2} \norm{\nabla \tv_{k+1} } ^2
        +\lambda \norm{\tv_{k+1}} ^2
        \le \hat\nu^{-1} \norm{ f }_{-1}^2 + \mu \norm{I_H(u)}^2,  \label{CDAPicard5}
    \end{equation}
which completes the proof.
\end{proof}

From \eqref{CDAPicard5} we can define $ C_{u,f} $ such that
\begin{equation}\label{eq: cuf}
    \norm{\nabla \tv_{k}} \leq \pare{ 2 \mu \hat{\nu}^{-1} \norm{ I_H (u) }^2 + 2\hat\nu^{-2}\norm{f}_{-1}^2  }^{1/2} := C_{u,f}.
\end{equation} 
We also define the data dependent parameter
\[
 \alpha = M \hat{\nu}^{-1}  C_{u,f}.  
\]

\begin{lemma}[Error of $ v_k - \tv_k$]\label{Lemma5}
We have the following error bound induced by Step 2 CDA-Newton using CDA-Picard solution as input. 

\begin{equation}
    \frac{\hat\nu}{4} \norm{\Grad (v_k - \tv_k)}^2 + \lambda \norm{ v_k - \tv_k }^2\leq \hat\nu \alpha^2 \norm{\Grad (v_k-v_{k-1}) }^2,
\end{equation}
    where 
\end{lemma}

\begin{proof}
    Subtracting \eqref{CDAPicard} from \eqref{CDANewton} both indexed at k results in
    \al
        \hat\nu (\Grad (v_k -\tv_k) , \Grad \chi) + \mu (I_H( v_k - \tv_k) , I_H \chi) 
        + b ( \tv_{k}, v_k - \tv_k, \chi ) + b(v_k - v_{k-1}, \tv_k, \chi)  = 0.
    \eal
Recall that $ e_k = v_k - v_{k-1}$ then choosing $ \chi = v_k -\tv_k$ yields 
    \al
        \hat\nu \norm{\Grad (v_k -\tv_k)}^2 + \mu \norm{ I_H( v_k -\tv_k) }^2
        &= -  b( e_k, \tv_k, v_k -\tv_k) .
    \eal
Applying \eqref{lowerboundlam} to the left hand side along with \eqref{eq: bbound2} and \eqref{eq: cuf}
to the right hand side yields the estimate 
    \al \frac{\hat\nu}{2} \norm{\Grad (v_k -\tv_k)}^2 + \lambda \norm{ v_k -\tv_k }^2 &\leq M \norm{\Grad e_k } \norm{\Grad \tv_k } \norm{\Grad v_k -\tv_k} \nonumber\\
    &\leq M C_{u,f} \norm{\Grad e_k } \norm{\Grad v_k -\tv_k} \nonumber\\
    &\leq \hat\nu\inv M^2 C_{u,f}^2 \norm{\Grad e_k }^2 + \frac{\hat\nu}{4} \norm{\Grad v_k -\tv_k}^2.
    \eal
Reducing gives
    \al \frac{\hat\nu}{4} \norm{\Grad( v_k -\tv_k)}^2 + \lambda \norm{ v_k -\tv_k }^2
    &\leq \hat\nu\inv M^2 C_{u,f}^2 \norm{\Grad e_k }^2.
    \eal
Substituting $ \alpha$ completes the proof.
\end{proof}



We now prove a convergence result for the new nonlinear solver.  This result shows the method converges quadratically, and that as $H$ decreases the convergence basin expands as its radius is scaled by $H^{-1/2}$.
\begin{theorem}[Convergence of CDA-Picard + CDA-Newton]
\label{Thm: CDA-Picard_res}
Assume $H<\frac{\hat\nu^{1/2}}{16\alpha^2C_I}$.  Then the following bound for differences between iterates of the CDA-Picard + CDA-Newton iteration holds:
\begin{align}
\| \nabla e_{k+1} \|^2_* 
 \le 32 M^2 C_I H \alpha^2 \hat\nu^{-2} \| e_k \|_*^2 \left(   4\alpha^2 \| \nabla e_k \|^2 +
128 C_I H \hat\nu^{-2} \alpha^4 \norm{e_k}^4_* \right) \nonumber \\
\le 512 M^2 C_I H \alpha^4 \hat\nu^{-3} \| e_k \|_*^4 
+ 4096 M^2 C_I^2 H^2 \alpha^6 \hat\nu^{-4} \| e_k \|_*^6.
\label{thm3}
\end{align}
\end{theorem}

\begin{proof}
We begin by subtracting CDA-Picard \eqref{CDAPicard} at step $k$ from CDA-Picard at step $k+1$.  Defining $\tilde e_k:= \tilde v_k - \tilde v_{k-1}$ and $e_k:= v_k - v_{k-1}$, we get that
    \begin{equation}
        \hat\nu (\nabla \tilde e_{k+1} , \nabla \chi) 
        + b(v_k, \tilde e_{k+1}, \chi) +  b(e_{k}, \tilde v_{k}, \chi) 
        +\mu (I_H(\tilde e_{k+1}),I_H\chi) 
        = 0, \: \forall v \in V. \label{CDAP1}
    \end{equation}
Choosing $ \chi = \tilde e_{k+1}$ vanishes the second left hand side term in \eqref{CDAP1} and  gives us
    \begin{equation}
        \hat\nu \norm{ \nabla \tilde e_{k+1} }^2 
        +\mu \norm{I_H(\tilde e_{k+1})}^2 
        = -b(e_{k}, \tilde v_{k}, \tilde e_{k+1}).  \label{CDAP2}
    \end{equation}
Applying \eqref{lowerboundlam} to the left hand side and \eqref{eq: bbound1}
to the right hand side yields the estimate 
    \begin{align}
        \frac{\hat\nu}{2} \norm{ \nabla \tilde e_{k+1} }^2 
        +\lambda \norm{\tilde e_{k+1}}^2
        & \leq M \norm{e_{k}}^{1/2} \norm{\Grad e_{k}}^{1/2} \norm{\Grad \tilde v_k} \norm{\Grad \tilde e_{k+1}} \nonumber \\
        & \le M C_{u,f} \norm{e_{k}}^{1/2} \norm{\Grad e_{k}}^{1/2}\norm{\Grad \tilde e_{k+1}}.\label{CDAP3}
    \end{align}
Applying Young's inequality on the right hand side term that
    \al
        \frac{\hat\nu}{2} \norm{ \nabla \tilde e_{k+1} }^2 
        +\lambda \norm{\tilde e_{k+1}}^2 
        & \le \hat\nu \alpha \norm{e_{k}}^{1/2} \norm{\Grad e_{k}}^{1/2}  \norm{\Grad \tilde e_{k+1}} \nonumber \\
        &\leq \alpha^2\hat\nu \norm{e_{k}} \norm{\Grad e_{k}} + \frac{\hat\nu}{4}  \norm{\Grad \tilde e_{k+1}}^2, \label{CDAP4}
    \eal
which simplifies to 
    \al
        \frac{\hat\nu}{4} \norm{ \nabla \tilde e_{k+1} }^2 
        +\lambda \norm{\tilde e_{k+1}}^2 
        \leq  \alpha^2\hat\nu \norm{e_{k}} \norm{\Grad e_{k}}. \label{CDAP5}
    \eal
Again applying Young's inequality to the right hand side of \eqref{CDAP5} produces the bound
    \al
        \frac{\hat\nu}{4} \norm{ \nabla \tilde e_{k+1} }^2 
        +\lambda \norm{\tilde e_{k+1}}^2 
        &\leq \alpha^2\hat\nu  \pare{   \frac{C_I H}{2}\norm{ \Grad e_{k}}^2 + \frac{1}{2C_I H} \norm{ e_k}^2 } \nonumber\\
        &\le  2 C_I H\alpha^2   \pare{  \frac{\hat\nu}{4} \norm{ \Grad e_{k}}^2 +  \lambda \norm{ e_k}^2 }. \label{CDAP6}
    \eal

Next, we consider differences in the Step 2 (CDA-Newton) iterates at step $k$ and $k+1$.  Subtracting \eqref{CDANewton} at step $k$ from \eqref{CDANewton} at step $k+1$ gives, for all $\chi\in V$,
\begin{multline}
\hat\nu (\nabla e_{k+1},\nabla \chi) 
+ \mu (I_H e_{k+1},I_H \chi)
+ \left( b(\tilde v_{k+1},v_{k+1},\chi) - b(\tilde v_k,v_k,\chi) \right) \\
+ \left( b(v_{k+1},\tilde v_{k+1},\chi) - b(v_k,\tilde v_k,\chi) \right)
- \left( b(\tilde v_{k+1},\tilde v_{k+1},\chi) - b(\tilde v_k,\tilde v_k,\chi) \right) =0. \label{CDAN1}
\end{multline}
Rewriting the $b$ term differences, we get for all $\chi\in V$ that
\begin{multline}
\hat\nu (\nabla e_{k+1},\nabla \chi) 
+ \mu (I_H e_{k+1},I_H \chi)
+ \left( b(\tilde v_{k+1},e_{k+1},\chi) + b(\tilde e_{k+1},v_k,\chi) \right) \\
+ \left( b(e_{k+1},\tilde v_{k+1},\chi) + b(v_k,\tilde e_{k+1},\chi) \right)
- \left( b(\tilde e_{k+1},\tilde v_{k},\chi) + b(\tilde v_{k+1},\tilde e_{k+1},\chi) \right) =0. \label{CDAN2}
\end{multline}
Choosing $\chi=e_{k+1}$ (which eliminates the first $b$ term), combining the second and fifth $b$ terms and then the fourth and sixth, and finally using \eqref{lowerboundlam} produces
\begin{multline}
\frac{\hat\nu}{2}\| \nabla e_{k+1} \|^2 
+ \lambda \| e_{k+1}\|^2
\le
| b(\tilde e_{k+1},\tilde v_{k} - v_k,e_{k+1})|  
+ | b(e_{k+1},\tilde v_{k+1},e_{k+1})| \\
+ | b(\tilde v_{k+1} - v_k,\tilde e_{k+1},e_{k+1})|. \label{CDAN3}
\end{multline}
We bound the three $b$ terms on the right hand side of \eqref{CDAN3} as follows.  For the first one, we use Lemma \ref{TRIL}, Young's inequality and \eqref{CDAP6} to get that
\begin{align}
| b(\tilde e_{k+1},\tilde v_{k} - v_k,e_{k+1})|
& \le M \| \nabla \tilde e_{k+1} \| \| \nabla (\tilde v_{k+1} - v_k) \| \| \nabla e_{k+1} \| \nonumber \\
& \le \frac{\hat\nu}{8} \| \nabla e_{k+1} \|^2 + 2 M^2  \| \nabla \tilde e_{k+1} \|^2 \| \nabla (\tilde v_{k} - v_k) \|^2 \nonumber \\
& \le \frac{\hat\nu}{8} \| \nabla e_{k+1} \|^2 + 16 M^2 C_I H \alpha^2 \hat\nu^{-2} \| e_k \|_*^2   \| \nabla (\tilde v_{k} - v_k) \|^2.\label{CDAN4a}
\end{align}
For the second $b$ term in \eqref{CDAN3}, we use use Lemma \ref{TRIL} and Young's inequality to obtain
\begin{align}
| b(e_{k+1},\tilde v_{k+1},e_{k+1})|
& \le M \| \nabla e_{k+1} \| \| \nabla \tilde v_{k+1}\| \| e_{k+1}) \|^{1/2} \| \nabla e_{k+1} \|^{1/2} \nonumber \\
& \le \alpha \hat\nu \| \nabla e_{k+1} \| \left( \|\nabla e_{k+1}\|^{1/2} \| e_{k+1} \|^{1/2} \right) \nonumber \\
& \le \frac{\hat \nu}{8} \| \nabla e_{k+1} \|^2 + 2\alpha^2 \hat \nu \|\nabla e_{k+1}\| \| e_{k+1} \| \nonumber \\
& \le \frac{\hat \nu}{8} \| \nabla e_{k+1} \|^2 + \frac{\lambda}{2} \| e_{k+1} \|^2 + 4\alpha^4 \lambda^{-1} \hat \nu \|\nabla e_{k+1}\|^2. \label{CDAN4b}
\end{align}
Finally for the third $b$ term we again use Lemma \ref{TRIL} and Young's inequality which results in
\begin{align}
| b(\tilde v_{k+1} - v_k,\tilde e_{k+1},e_{k+1})|
& \le M \| \nabla (\tilde v_{k+1} - v_k) \| \| \nabla \tilde e_{k+1} \|  \| \nabla e_{k+1} \| \nonumber \\
& \le \frac{\hat\nu}{8} \| \nabla e_{k+1} \|^2 + 2\hat\nu^{-1} M^2 \| \nabla (\tilde v_{k+1} - v_k) \|^2 \| \nabla \tilde e_{k+1} \|^2 \nonumber \\
& \le \frac{\hat\nu}{8} \| \nabla e_{k+1} \|^2 + 16 \hat\nu^{-2} M^2 C_I H \alpha^2 \| \nabla (\tilde v_{k+1} - v_k) \|^2 \| e_k \|_*^2. \label{CDAN4c}
\end{align}
Combining \eqref{CDAN3} with bounds \eqref{CDAN4a}-\eqref{CDAN4c}, we get that 
\begin{multline}
\left( \frac{\hat\nu}{8} - 8\alpha^4 C_I^2 H^2 \right) \| \nabla e_{k+1} \|^2 
+ \frac{\lambda}{2} \| e_{k+1}\|^2
\le \\
16 M^2 C_I H \alpha^2 \hat\nu^{-2} \| e_k \|_*^2 \left(   \| \nabla (\tilde v_{k} - v_k) \|^2 +
\| \nabla (\tilde v_{k+1} - v_k) \|^2 \right). \label{CDAN5}
\end{multline}
Using the assumption on $H$ and applying Lemma \ref{Lemma5} 
now gives the estimate
\begin{align}
 \frac{\nu}{2} \| \nabla e_{k+1} \|^2 
 + \frac{\lambda}{2} \| e_{k+1}\|^2 
 \le 16 M^2 C_I H \alpha^2 \hat\nu^{-2} \| e_k \|_*^2 \left(   4\alpha^2 \| \nabla e_k \|^2 +
\| \nabla (\tilde v_{k+1} - v_k) \|^2 \right). \label{CDAN6}
\end{align}

It remains to bound $\| \nabla (\tilde v_{k+1} - v_k) \|^2$.  To do this, we subtract \eqref{CDAPicard} from \eqref{CDANewton} with index k, which gives for all $\chi\in V$,
   \al
        \hat\nu(\Grad (\tv_{k+1} - v_k ) , \Grad \chi)  +\mu (I_H( \tv_{k+1} - v_k) , I_H(\chi))
        +b(v_k,\tv_{k+1} - \tv_k,\chi)  + b(\tv_k,\tv_k - v_k,\chi)
        = 0.     \nonumber 
    \eal
Choosing $ \chi = \tv_{k+1} - v_k$ yields, by adding and subtracting $\tilde v_k$ in the first argument of the first $b$ term and reducing, 
   \al
        \hat\nu\norm{\Grad (\tv_{k+1} - v_k)}^2  &+\mu \norm{(I_H( \tv_{k+1} - v_k)}^2 \nonumber \\
        & =
        -b(v_k,\te_{k+1},\tv_{k+1} - v_k) +b(\tv_k,v_k -\tv_k,\tv_{k+1} - v_k)  \nonumber \\
        &= b(\tv_k - v_k, \te_{k+1}, \tv_{k+1} - v_k) + b(\tv_{k},v_k -\tv_k - \te_{k+1} , \tv_{k+1} - v_k ) \nonumber \\
        &= -b(v_k -\tv_k, \te_{k+1},\tv_{k+1} - v_k)- b(\tv_k,\tv_{k+1} - v_k,\tv_{k+1} - v_k).
    \eal
The second term on the right-hand side vanishes, and then applying \eqref{lowerboundlam} to the left hand side and \eqref{eq: bbound1}
to the right hand sides yields the estimate
   \al
        \frac{\hat\nu}{2}\norm{\Grad (\tv_{k+1} - v_k)}^2  +\lambda \norm{ \tv_{k+1} - v_k}^2 
        &\leq \norm{\Grad (v_k -\tv_k)} \norm{\Grad \te_{k+1}}\norm{\Grad (\tv_{k+1} - v_k)}. \label{eq: eP3}
    \eal
Applying Young's inequality to the right hand sides gives
   \al
        \frac{\hat\nu}{2}\norm{\Grad (\tv_{k+1} - v_k)}^2  +\lambda \norm{ \tv_{k+1} - v_k}^2 
        &\leq \hat\nu\inv \norm{\Grad (v_k -\tv_k)}^2 \norm{\Grad \te_{k+1}}^2 + \frac{\hat\nu}{4}\norm{\Grad (\tv_{k+1} - v_k)}^2. \label{eq: eP4}
    \eal   
Combining like terms and using the bound from Lemma \ref{Lemma5}, and the error bound from \eqref{CDAP6} results in
   \al
        \frac{\hat\nu}{4}\norm{\Grad (\tv_{k+1} - v_k)}^2  +\lambda \norm{ \tv_{k+1} - v_k}^2 
        &\leq 4\alpha^2\norm{\Grad e_k}^2 ( 8C_I H\hat\nu\inv \alpha^2 \norm{e_k}^2_* )\nonumber \\
        &= 32C_I H \hat\nu\inv \alpha^4 \norm{e_k}^4_*, \label{eq: eP5}
    \eal
which implies
\[
\norm{\Grad (\tv_{k+1} - v_k)}^2 \le 128 C_I H \hat\nu^{-2} \alpha^4 \norm{e_k}^4_*.
\]
Using this in \eqref{CDAN6} now provides
\begin{align}
\| \nabla e_{k+1} \|^2_* 
 \le 32 M^2 C_I H \alpha^2 \hat\nu^{-2} \| e_k \|_*^2 \left(   4\alpha^2 \| \nabla e_k \|^2 +
128 C_I H \hat\nu^{-2} \alpha^4 \norm{e_k}^4_* \right) \nonumber \\
\le 512 M^2 C_I H \alpha^4 \hat\nu^{-3} \| e_k \|_*^4 
+ 4096 M^2 C_I^2 H^2 \alpha^6 \hat\nu^{-4} \| e_k \|_*^6.
\label{CDAN7}
\end{align}
This completes the proof.
\end{proof}

\section{Modified Newton parameter recovery algorithm} \label{sec:paramrec}

In this section we derive the modified Newton parameter recovery algorithm to find the viscosity $\nu$ from given partial solution data $I_H u$, where $u$ is the NSE velocity solution satisfying \eqref{eq: weak_NSE3}.  We begin by assuming that for given partial solution data $I_H u$, solutions $v=v(\hat\nu)\in V$ to \eqref{cdanse} depend smoothly on the viscosity $\hat\nu$.  Noting that $u=v(\nu)$, we use Taylor series to obtain the relation
\[
v(\hat \nu) = v(\nu) + (\nu - \hat \nu) \left( \begin{array}{c} v_1'(c_1) \\ v_2'(c_2) \\ v_3'(c_3) \end{array} \right) =  v(\nu) + (\nu - \hat \nu)w = u + (\nu - \hat \nu)w,
\]
where $c_1,c_2,c_3$ are between $\nu$ and $\hat\nu$, $v_i$ are the components of $v$, and we have denoted $w:=\left( \begin{array}{c} v_1'(c_1) \\ v_2'(c_2) \\ v_3'(c_3) \end{array} \right).$

Subtracting the CDA-NSE \eqref{cdanse} for arbitrary fixed $\hat\nu$ that is sufficiently close to $\nu$, from the NSE \eqref{eq: weak_NSE3}, gives us 
\[
\hat \nu (\nabla e,\nabla \chi) + (\nu-\hat \nu)(\nabla u,\nabla \chi) + b(v,e,\chi) + b(e,u,\chi) + \mu(I_H e,I_H \chi) = 0 \ \forall \chi \in V,
\]
where $e = e(\hat \nu) = u - v(\hat\nu)$.  Taylor series then yields
\[
e = (\hat \nu - \nu)w.
\]
Taking $\chi=e$ vanishes the first nonlinear term since $b(v,e,e)=0$, and yields the equation
\[
\hat \nu \| \nabla e \|^2 + (\nu - \hat\nu)(\nabla u,\nabla e)  + b(e,u,e) + \mu \| I_H e \|^2 = 0.
\]
Now substituting for each $e$ except in the last term gives the equation
\[
\hat \nu (\hat\nu - \nu)^2 \| \nabla w \|^2 - (\nu - \hat\nu)^2(\nabla u,\nabla w)  + (\hat\nu - \nu)^2 b(w,u,w) + \mu \| I_H e \|^2 = 0.
\]
Thus we can write
\[
 \| I_H e \|^2 = C (\hat\nu - \nu)^2,
\]
where $C=\mu^{-1} \left( \hat \nu  \| \nabla w \|^2 - (\nabla u,\nabla w) + b(w,u,w)\right).$   While we cannot rule out the possibility of $C=0$, we believe this can happen only in isolated (and perhaps diabolical) cases such as when $w=0$; there were no issues in any of our numerical tests.  We have thus established that, for $C\ne 0$, 
\begin{equation}
E(\hat\nu) := \frac12 \| I_H e(\hat\nu) \|^2 = C (\hat\nu - \nu)^2, \label{E1}
\end{equation}
which implies that near $\nu$, $E$ is a parabola with a single root at the vertex $(\nu,0)$.  

While the relation \eqref{E1} alone is useful and can already be exploited to create a method for finding $\nu$, a significant improvement results from determining $E'(\hat\nu)$.  By implicitly differentiating \eqref{cdanse} with respect to $\hat\nu$, we obtain for all $\chi\in V$ that
\[
 (\nabla v,\nabla\chi) + \hat\nu\left(\nabla \frac{\partial v}{\partial \hat\nu},\nabla\chi \right) + b^*\left(\frac{\partial v}{\partial \hat\nu},v,\chi\right) + b^*\left(v,\frac{\partial v}{\partial \hat\nu},\chi\right) + \mu \left(I_H \frac{\partial v}{\partial \hat\nu},I_H\chi\right) = 0.
 \]
Since $v$ is considered known, we have the following {\it linear} system to determine $\frac{\partial v}{\partial \hat\nu}\in V$: for all $\chi\in V$,
\begin{equation}
 \hat\nu\left(\nabla \frac{\partial v}{\partial \hat\nu},\nabla\chi \right) + b^*\left(\frac{\partial v}{\partial \hat\nu},v,\chi\right) + b^*\left(v,\frac{\partial v}{\partial \hat\nu},\chi\right) + \mu \left(I_H \frac{\partial v}{\partial \hat\nu},I_H\chi\right) = -(\nabla v,\nabla\chi) . \label{dv}
\end{equation}
Hence for a given $v=v(\hat\nu)\in V$, just one linear solve similar to a CDA-Newton solve determines $\frac{\partial v}{\partial \hat\nu}\in V$.  Hence we can now calculate the derivative
\[
\frac{\partial E(\hat\nu)}{\partial \hat\nu} = \left(I_H \frac{\partial v}{\partial \hat\nu},I_H e\right),
\]
for which we have already computed the right hand side, and thus we obtain $\frac{\partial E(\hat\nu)}{\partial \hat\nu}$ at the cost of a single linear solve.

We now have all the pieces to devise a modified Newton root-finding algorithm.  Since $E$ is parabolic and $\frac{\partial E(\hat\nu)}{\partial \hat\nu}(\nu)=0$, this implies the order of the root is two, and so the appropriate modified Newton method for quadratic convergence is to guess $\nu_0$ and then iterate:
\begin{equation}
\nu_{k+1} = \nu_k - \frac{ 2E(\nu_k) }{ \frac{\partial E(\nu)}{\partial \hat\nu}(\nu_k)}. \label{newt1}
\end{equation}
One step of this modified Newton algorithm requires calculating $E(\nu_k)$ and $\frac{\partial E(\hat\nu)}{\partial \hat\nu}(\nu_k)$, which requires knowing $v(\nu_k)$ and $\frac{\partial v}{\partial \hat\nu}(\nu_k)$.
Determining $v(\nu_k)$ requires a nonlinear solve of \eqref{cdanse}, for which we propose using CDA-Picard+CDA-Newton of the previous section (our numerical tests show that this takes only a few iterations to converge on average).  Determining $\frac{\partial v}{\partial \hat\nu}(\nu_k)$ requires one linear solve of \eqref{dv}, which not coincidentally is very similar in form to a CDA-Newton solve.  
We have thus derived the following algorithm for determining an unknown viscosity $\nu$ from partial NSE solution data $I_Hu$.

\begin{algorithm}[Modified Newton for recovering unknown viscosity from solution data]\label{modnewt}
Given partial solution data $I_Hu$, tolerance $tol$ and maximum number of iterations $maxit$,  do:
\begin{enumerate}
\item Guess $\nu_0$ and $v_0\in V$.
\item For $i=1:maxit$
\begin{itemize}
\item Solve the CDA-NSE nonlinear system  \eqref{cdanse} using $\hat\nu=\nu_{i-1}$ to determine the velocity solution $v_i=v.$ 
\begin{itemize}
\item Use CDA-Picard+CDA-Newton proposed in Section 4 as the nonlinear solver, with $v_{i-1}$ as the initial guess.
\end{itemize}
\item Solve the linear system \eqref{dv} using $\hat\nu=\nu_{i-1}$ and $v=v_i$ to determine $\omega_i=\frac{\partial v}{\partial \hat\nu}$.
\item Update the viscosity as
\[
\nu_i = \nu_{i-1} - \frac{ \| I_H \left(v_i - u\right) \|^2 }{\left(I_H \omega_i, I_H (v_i - u)\right) }.
\]
\item If $|\nu_i - \nu_{i-1}|<tol$, break.
\end{itemize}

\end{enumerate}

\end{algorithm}

\section{Numerical Results} \label{sec: Numres}

We now demonstrate the effectiveness of the modified Newton viscosity recovery algorithm and the CDA-Picard+CDA-Newton nonlinear solver.  We use three benchmark tests: 2D driven cavity, 3D driven cavity, and 2D channel flow through a sudden expansion.  All tests reveal the fast and robust convergence of the new methods.  We also consider the case of noisy data, even though this is beyond the theory developed here.  We find that the modified Newton viscosity recovery algorithm still quickly converges, but only up to a noise-dependent range near the solution.

All of our tests use divergence-free $(P_k^d,P_{k-1}^{disc})$ Scott-Vogelius (SV) finite elements with $k=d$, which we denote by $(X_h,Q_h)\subset (H^1(\Omega),L^2(\Omega))$. The meshes used are barycenter refined triangular/tetrahedral meshes, where these elements are known to be stable \cite{JLMNR17,arnold:qin:scott:vogelius:2D,Z05}.  Using SV elements provides strong mass conservation, optimal accuracy, and a pressure robust discretization \cite{JLMNR17,CELR11}.  

The discretely divergence-free subspace is defined by
\[
V_h := \{  v\in X_h,\ (\nabla \cdot v,q)=0\ \forall q\in Q_h \},
\] 
and since for these elements it holds that $\nabla \cdot X_h\subset Q_h$, we have that 
\[
V_h := \{  v\in X_h,\ \| \nabla \cdot v_h \|=0 \}.
\] 
Hence $V_h\subset V$, and all the analysis in this paper (which is performed in $V$) extends trivially to the discrete setting using $V_h$.  Note that for our tests, true solutions come from direct numerical simulations and using continuation methods when needed.  That is, we assume that true solutions satisfy: $u \in V_h$ solves
\al
\nu \pare{ \Grad u, \Grad \chi } + b(u,u,\chi)  &= \langle f,\chi \rangle.\  \forall \chi \in V_h. \label{eq: disc_NSE3} 
\eal
It is known that for these solutions and element choices, $\| u - u_{nse} \|_{L^2}=O(h^{d+1})$.

Each iteration of the modified Newton viscosity recovery algorithm, Algorithm \ref{modnewt}, requires solving CDA-NSE using the CDA-Picard+CDA-Newton.  For both CDA-Picard and CDA-Newton, we use
Augmented Lagrangian type preconditioning of GMRES following \cite{benzi,HR13,CLLRW13,BB12} to perform the linear solves, with the augmentation created from grad-div stabilization. The inner solves use direct solvers.  This linear solver approach is found to be very effective in all of our tests, and the number of GMRES iterations is typically less than 10.

\subsection{Performance of the modified Newton viscosity recovery algorithm}

We consider three tests to illustrate the effectiveness of the proposed modified Newton viscosity recovery algorithm, i.e. Algorithm \ref{modnewt}.  For these tests, true solutions are assumed to satisfy \eqref{eq: disc_NSE3}.  There are two tolerances involved in the accuracy of such solutions: the nonlinear solver tolerance and the linear solver tolerance.  We use Augmented Lagrangian preconditioned GMRES as a linear solver as in \cite{benzi,HR13,CLLRW13,BB12}, with tolerance $10^{-10}$.  For nonlinear solvers, we used an $H^1$ nonlinear solver tolerance of $10^{-8}$ in 2D and $10^{-6}$ in 3D.  Note the nonlinear solver tolerance is limited by the linear solver tolerance, and also the dimension $d$ and Reynolds number since they affect conditioning of the matrices.

\subsubsection{2D driven cavity}

\begin{figure}[h!]
\center
$Re=5,000$ \hspace{1in}  $Re=10,000$ \\
\includegraphics[width = .3\textwidth, height=.28\textwidth,viewport=115 45 465 390, clip]{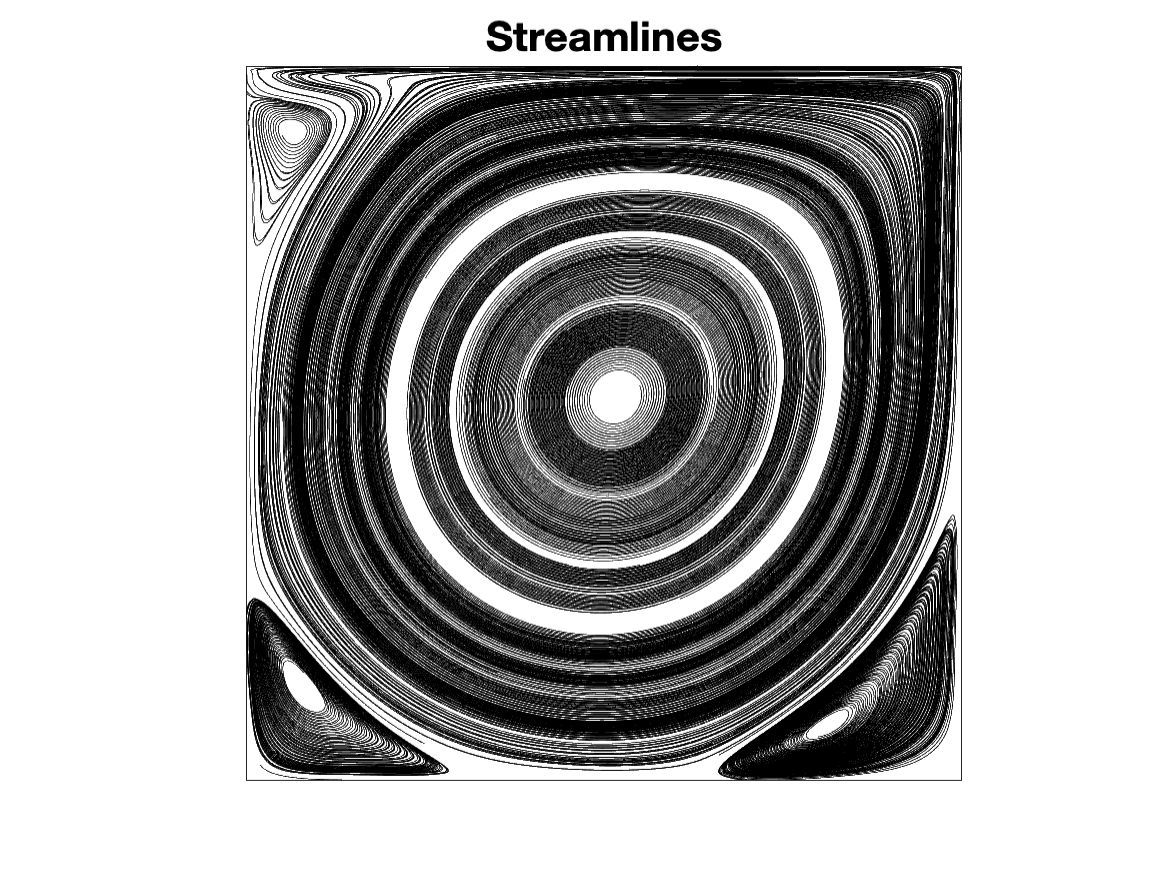}  
\includegraphics[width = .3\textwidth, height=.28\textwidth,viewport=115 45 465 390, clip]{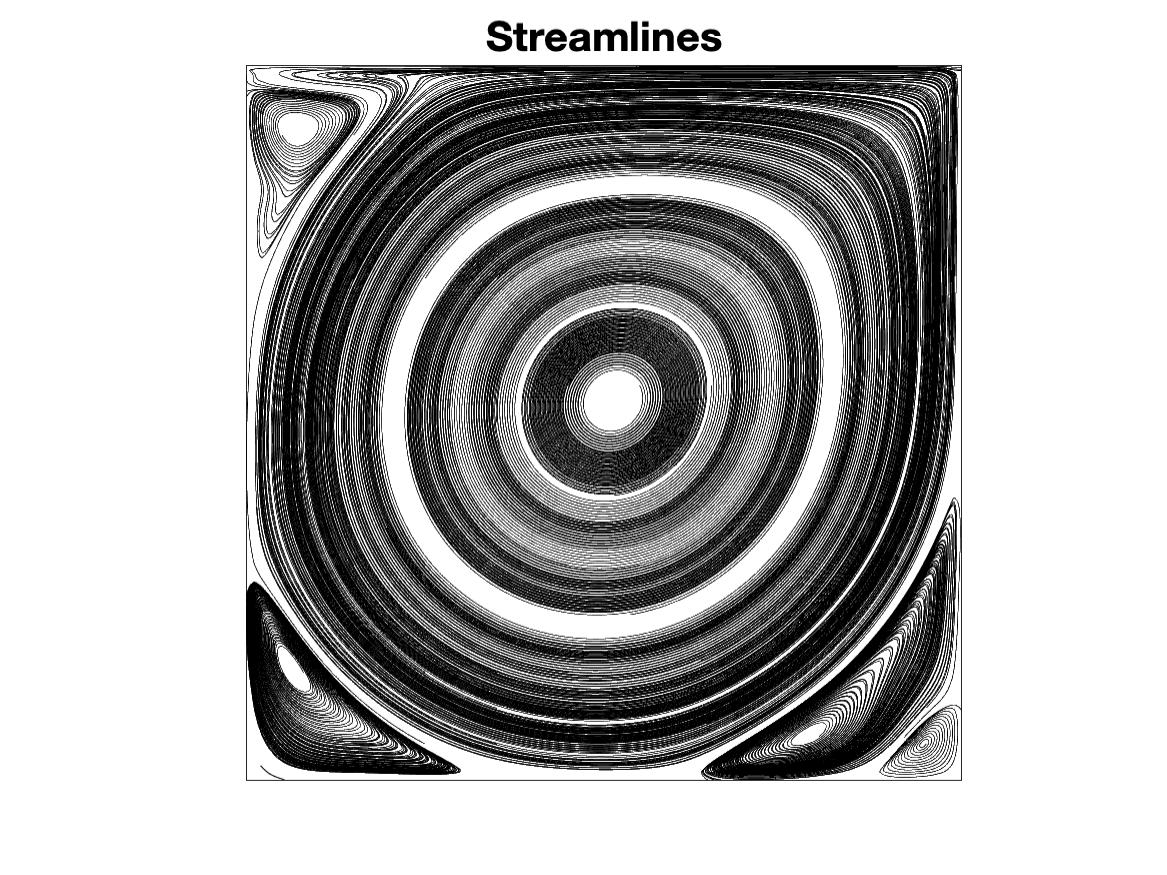}  
\caption{\label{dc2d} The plots above show 2D driven cavity velocity solution streamlines for varying $Re$.}
\end{figure}

Our first test for Algorithm \ref{modnewt} is the 2D driven cavity benchmark test for incompressible NSE flow \cite{ghia:cavity,bruneau:cavity,ECG05}.  The domain is $\Omega=(0,1)^2$, there is no forcing $f=0$, homogeneous Dirichlet velocity boundary conditions are prescribed on the sides and bottom (the walls), and $[ 1,0 ]^T$ is prescribed on the top (the lid).  The Reynolds number for this problem is defined to be the inverse of the kinematic viscosity, $Re:=\nu^{-1}$, and we choose for our tests $Re$=5000 and 10000.  Our discretization provides 680K total  degrees of freedom (dof), and true solutions for these $Re$ are shown in Figure \ref{dc2d} and are in agreement with solutions in the literature \cite{ECG05}.  CDA parameters are chosen to be $\mu=1$ and $H=\frac{1}{16}$; thus we use true solution data at just 256 nodes (i.e. only 512 total dof out of the total 680K dof).  The tolerance for Algorithm \ref{modnewt} is set as $10^{-7}$.

\begin{table}[h!]
\centering
\begin{tabular}{|c||c|c||c|c|}
	\hline
Iteration &	$\nu_k$ &   \# CDA-P + CDA-N iters &	$\nu_k$ &   \# CDA-P + CDA-N iters  \\ \hline
0 (initial guess) & 3.33333e-4 & - & 1.00000e-4 & - \\ \hline
1 & 1.90021e-4 &  6 & 1.99750e-4 & 6 \\ \hline
2 & 1.99936e-4  & 4& 1.99999e-4   & 4 \\ \hline
3 & 2.00000e-4 &  3& 2.00002e-4 &  3 \\ \hline
\end{tabular}
\caption{\label{convNewton2D} Shown above are convergence results for Algorithm \ref{modnewt} and for CDA-Picard+CDA-Newton for the 2D cavity, using true $\nu=\frac{1}{5000}$ ($Re$=5000) and initial guesses of $\frac{1}{3000}$ and $\frac{1}{10000}$.}
\end{table}

We first test Algorithm \ref{modnewt} for Reynolds number $Re$=5000, i.e. kinematic viscosity $\nu=\frac{1}{5000}$, with varying initial guesses for the viscosity.  Results are shown in Table \ref{convNewton2D} for both the convergence of Algorithm \ref{modnewt} and of CDA-Picard+CDA-Newton.  We observe rapid convergence of the modified Newton viscosity recovery algorithm, for both initial guesses of $\nu_0=\frac{1}{3000}$ and $\frac{1}{10000}$ convergence is obtained in just 3 iterations.
CDA-Picard+CDA-Newton is also observed to perform very well: for both runs, it needed 6, 4, and 3 iterations to converge the CDA-NSE system.

We also test Algorithm \ref{modnewt} for Reynolds number $Re$=10000, i.e. kinematic viscosity $\nu=\frac{1}{10000}$, with varying initial guesses for the viscosity.  This is a harder problem, and although the literature is not settled on whether this is past the first Hopf birfurcation \cite{bruneau:cavity,CHOR17}, steady solutions still exist at this $Re$ even if time dependent ones do not converge to those steady solutions \cite{G94,Layton08}.  For initial guesses $\nu_0=\frac{1}{4000}$ and $\frac{1}{13000}$ we observe rapid convergence of Algorithm \ref{modnewt} to the true viscosity, needing just 4 and 3 iterations to converge, respectively.  We also observe fast convergence of the CDA-Picard+CDA-Newton nonlinear solver.  The most iterations CDA-Picard+CDA-Newton takes is 7 iterations when $\nu_{i-1}=\frac{1}{13000}$ and $v_0=0$, but even this is very impressive convergence since stabilized Newton does not converge for the steady NSE when $\nu\le \frac{1}{3500}$. \cite{PRTX25}.

\begin{table}[h!]
\centering
\begin{tabular}{|c||c|c||c|c|}
	\hline
Iteration &	$\nu_k$ &   \# CDA-P + CDA-N iters &	$\nu_k$ &   \# CDA-P + CDA-N iters  \\ \hline
0 (initial guess) & 2.50000e-4  & - & 7.69231e-5 & - \\ \hline
1 & 7.70815e-5 &  5 & 9.95502e-5 & 7 \\ \hline
2 & 9.95587e-5  & 5& 9.99999e-5  & 4 \\ \hline
3 & 9.99999e-5  &  4& 1.000006-4 &  3 \\ \hline
4 & 1.00001e-4 &  3&   &    \\ \hline
\end{tabular}
\caption{\label{convNewton2Db} Shown above are convergence results for Algorithm \ref{modnewt} and for CDA-Picard+CDA-Newton for the 2D cavity, using true $\nu=\frac{1}{10000}$ ($Re$=10000) and initial guesses $\frac{1}{4000}$ and $\frac{1}{13000}$.}
\end{table}

\subsection{3D driven cavity}

\begin{figure}[h!]
\centering
$Re$=1000\\
\includegraphics[width=0.9\textwidth]{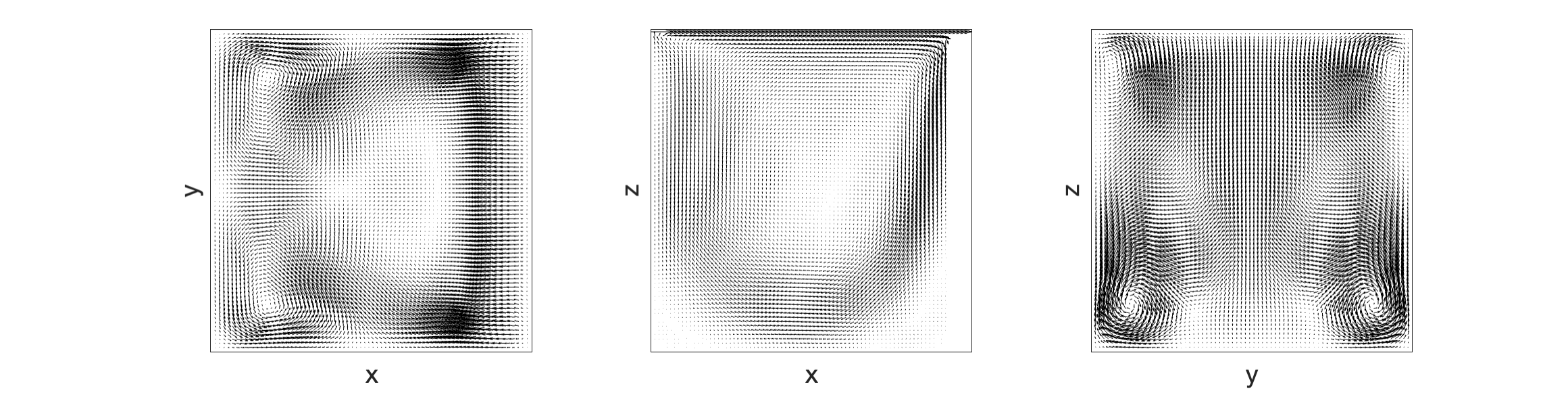}\\
\caption{\label{plots3d}Shown above are midsliceplanes for the $Re$=1000 velocity solution found for the 3D driven cavity.}
\end{figure}

\begin{table}[h!]
\centering
\begin{tabular}{|c||c|c||c|c|}
	\hline
Iteration &	$\nu_k$ &   \# CDA-P + CDA-N iters &	$\nu_k$ &   \# CDA-P + CDA-N iters  \\ \hline
0 (initial guess) & 1.000000e-2 & - & 2.500000e-3 & - \\ \hline
1 & 5.030392e-3 &  4 & 5.068201e-3 & 4 \\ \hline
2 & 5.000001e-3  & 3& 5.000012e-3  & 3 \\ \hline
3 & 4.999998e-3 &  3& 4.999999e-3 &  2 \\ \hline
\end{tabular}
\caption{\label{convNewton2Dc}  Shown above are convergence results for Algorithm \ref{modnewt} and for CDA-Picard+CDA-Newton for the 3D cavity, using true $\nu=\frac{1}{200}$ ($Re$=200) and initial guesses of $\frac{1}{100}$ and $\frac{1}{400}$.}
\end{table}

\begin{table}[h!]
\centering
\begin{tabular}{|c||c|c||c|c|}
	\hline
Iteration &	$\nu_k$ &   \# CDA-P + CDA-N iters &	$\nu_k$ &   \# CDA-P + CDA-N iters  \\ \hline
0 (initial guess) & 2.500000e-3 & - & 6.666667e-4 & - \\ \hline
1 & 0.999359e-3 &  4 & 1.004666e-3 & 5 \\ \hline
2 & 1.000000e-3  & 4& 1.000000e-3  & 4 \\ \hline
3 & 1.000000e-3 &  2& 1.000000e-3 &  3 \\ \hline
\end{tabular}
\caption{\label{convNewton2Dc2}  Shown above are convergence results for Algorithm \ref{modnewt} and for CDA-Picard+CDA-Newton for the 3D cavity, using true $\nu=\frac{1}{1000}$ ($Re$=1000) and initial guesses of $\frac{1}{400}$ and $\frac{1}{1500}$. }
\end{table}

Our second test problem is the 3D driven cavity, which is an analogue to the 2D test problem above.  Here, 
$\Omega=(0,1)^3$, $f=0$, and no slip velocity boundary conditions are strongly enforced on the walls and 
 $[1, 0, 0]^T$ on the `lid'.  Our tests use true solutions for $Re$=200 and 1000, and
 a plot of the $Re$=1000 solution is shown in Figure \ref{plots3d} that agrees well with the literature \cite{WB02}.  Our discretization provided for 796K total dof, and for CDA we used $\mu=1$ for all tests and $H=\frac{1}{10}$ for $Re$=200 and $H=\frac{1}{16}$ for $Re$=1000.
   
Results of Algorithm \ref{modnewt} are shown in Table \ref{convNewton2Dc} for $Re$=200 and Table \ref{convNewton2Dc2} for $Re$=1000.  In all cases, the CDA-Picard+CDA-Newton solver performed very well, yielding rapid and reliable convergence.  Algorithm \ref{modnewt} performed very well for $Re$=200 for both initial guesses of $\frac{1}{100}$ and $\frac{1}{400}$, converging in just 3 iterations to the true viscosity.  

\subsubsection{2D channel flow through an expansion}

\begin{figure}[h!]
\begin{center}
\includegraphics[width = .8\textwidth, height=.16\textwidth,viewport=170 20 1300 250, clip]{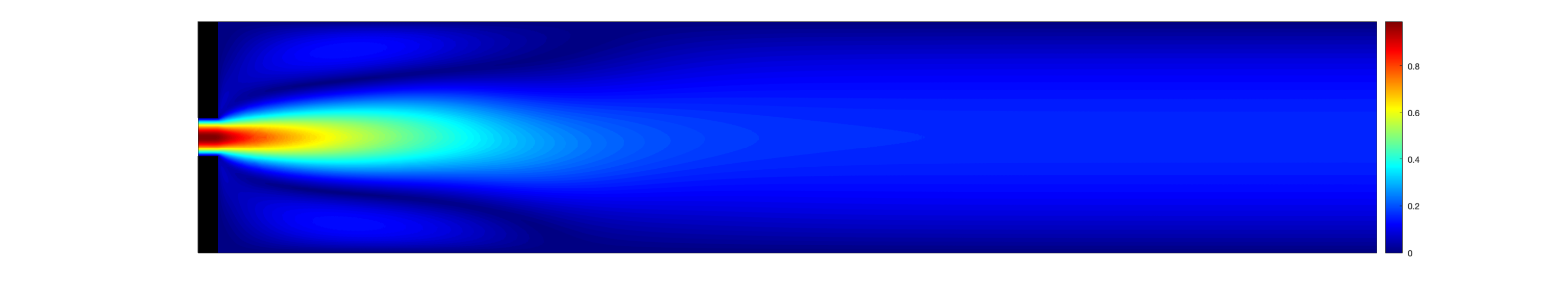}\\
\includegraphics[width = .8\textwidth, height=.16\textwidth,viewport=170 20 1300 250, clip]{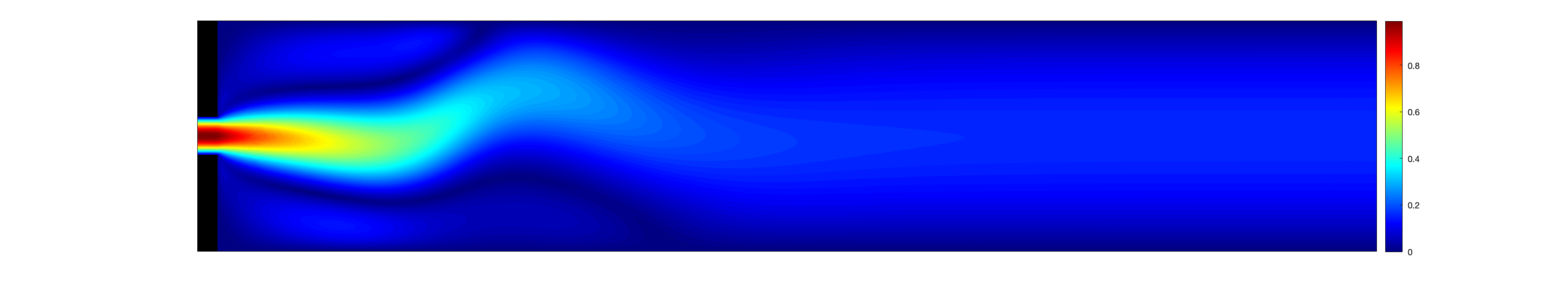}\\
\includegraphics[width = .8\textwidth, height=.16\textwidth,viewport=170 20 1300 250, clip]{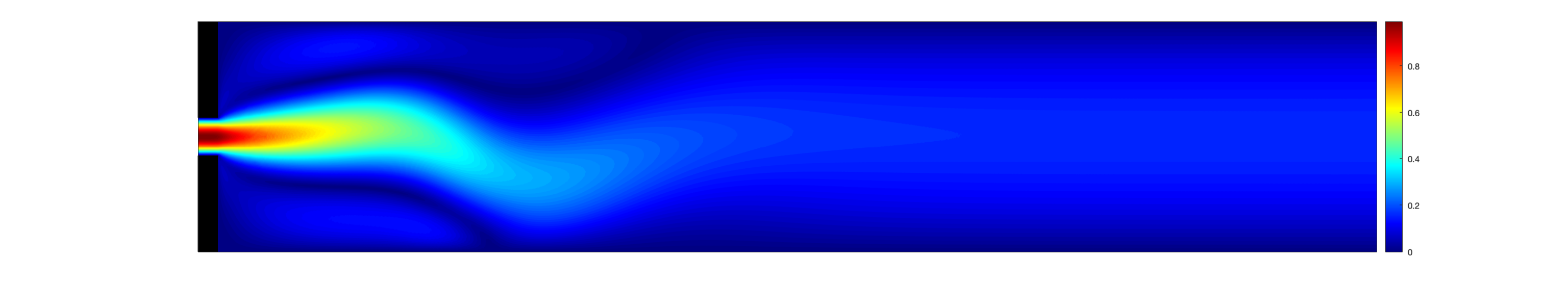}\\
\includegraphics[width = .8\textwidth, height=.16\textwidth,viewport=170 20 1300 250, clip]{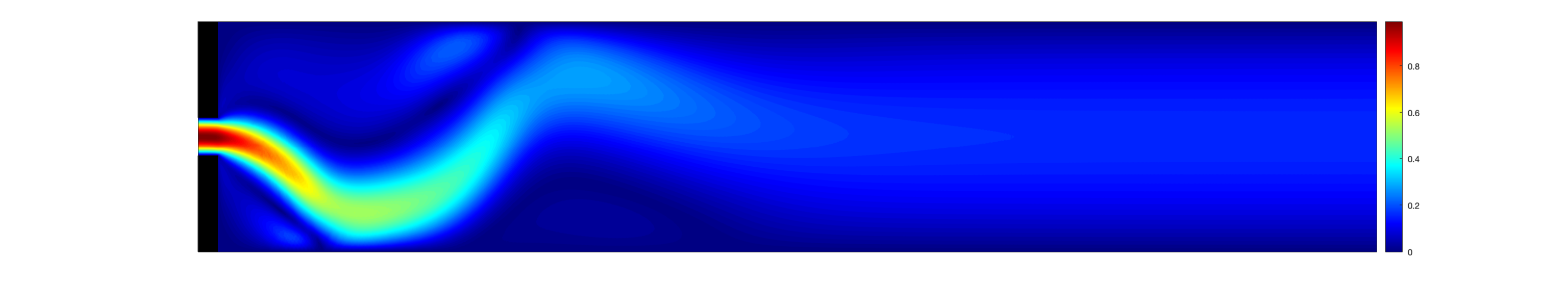}\\
\includegraphics[width = .8\textwidth, height=.16\textwidth,viewport=170 20 1300 250, clip]{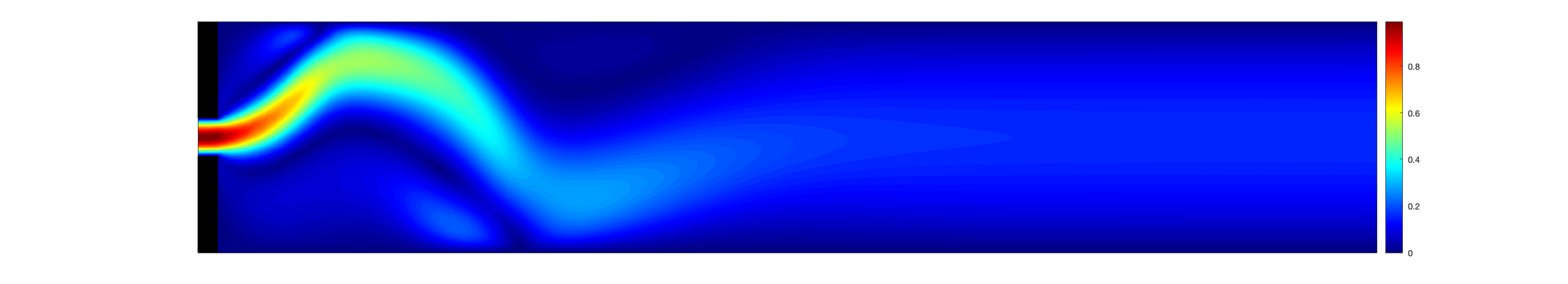}
\caption{\label{N1}  Shown above are the 5 solutions found by the NGMRES deflation method for the NSE channel expansion problem. }
\end{center}
\end{figure}

\begin{figure}[h!]
\begin{center}
\includegraphics[width = .8\textwidth, height=.16\textwidth,viewport=170 20 1450 280, clip]{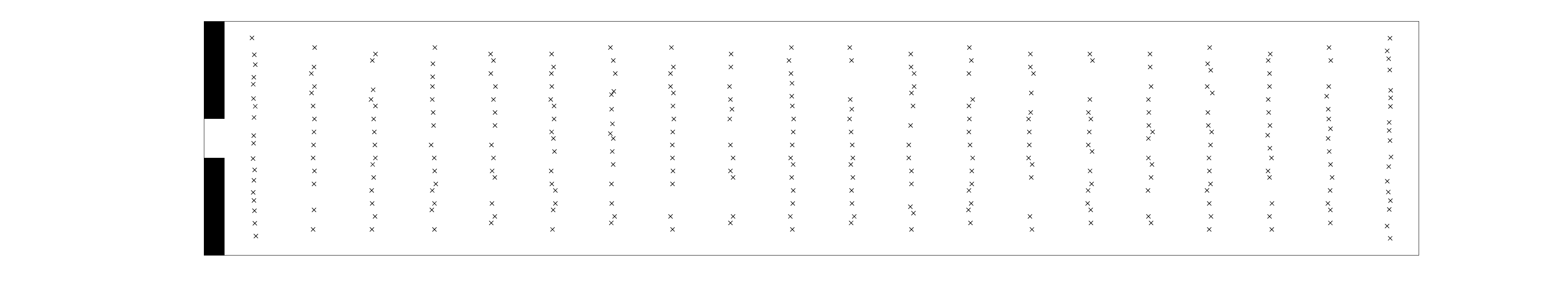}\caption{\label{IHnodes}  Shown above are the locations where the true solution is known for the 2D channel flow with expansion. }
\end{center}
\end{figure}

\begin{table}[h!]
\centering
\begin{tabular}{|c||c|c||c|c|}
	\hline 
	\multicolumn{5}{|c|}{Recovery of $\nu$ using Solution 1}\\ \hline
Iteration &	$\nu_k$ &   CDA-P + CDA-N iters  & $\nu_k$ & CDA-P + CDA-N iters \\ \hline
0 (initial guess) & 0.050000000   & - & 0.010000000 & -  \\ \hline
1 & 0.011693389 &  4 & 0.019351005 &6 \\ \hline
2 & 0.019519756 &  5 & 0.019997315&4\\ \hline
3 & 0.019998669 &  4 & 0.020000004&3\\ \hline
4 & 0.020000001 &  3 & 0.020000000&2\\ \hline
5 & 0.020000000 &  2 &  &  \\ \hline
\end{tabular} \\ \ \\
\begin{tabular}{|c||c|c||c|c|}
	\hline 
	\multicolumn{5}{|c|}{Recovery of $\nu$ using Solution 2}\\ \hline
Iteration &	$\nu_k$ &   CDA-P + CDA-N iters  & $\nu_k$ & CDA-P + CDA-N iters \\ \hline
0 (initial guess) & 0.050000000   & - & 0.010000000 & -  \\ \hline
1 & 0.012267192 &  4 & 0.018733355 &7 \\ \hline
2 & 0.019266767 &  5 & 0.019980492&4\\ \hline
3 & 0.019931677 &  4 & 0.019999975&3\\ \hline
4 & 0.019999993 &  3 & 0.020000000&2\\ \hline
5 & 0.020000000 &  2 &  &  \\ \hline
\end{tabular} 
\caption{\label{convNewtonexp} Shown above are convergence results for Algorithm \ref{modnewt} and for CDA-Picard+CDA-Newton for the 2D channel expansion with two distinct true solutions., using true $\nu=\frac{1}{50}$ ($Re$=50) and initial guesses of $\frac{1}{20}$ and $\frac{1}{100}$.}
\end{table}

For our third test problem, we consider a channel expansion problem studied by Farrell et al \cite{FBF15}, which is known to admit multiple solutions for sufficiently large $Re$.  The domain is $\Omega=(0,2.5]\times (-1,1) \cup (2.5,150)\times(-6,6),$ which represents a channel of width 2 entering into a channel of width 12.  The Reynolds number is taken to be $Re = \nu^{-1}$ for this problem, parabolic inflow and no-slip on the walls are enforced as Dirichlet boundary conditions for velocity, and zero-traction is weakly enforced at the outflow.  We choose $Re$=50, and for this $Re$ it is known that the NSE admits 5 distinct solutions \cite{FBF15}, which are shown in Figure \ref{N1}.  We refer to Solution 1 as the top solution in the figure, and Solution 2 as the second solution from top in the figure.  We compute on a barycenter refined Delaunay mesh that yields 97K total dof with SV elements.

We now test Algorithm \ref{modnewt}'s ability to recover the viscosity using partial solution data from this multi-solution test problem.  We first test parameter recovery using Solution 1 and initial guesses $\nu_0=\frac{1}{20}$ and $\frac{1}{100}$, and then repeat the test with Solution 2.  For CDA, we take $\mu=1$ and use 
297 nodes as measurement locations to construct $I_Hu$, see Figure \ref{IHnodes}.  Results are shown in Table \ref{convNewtonexp}.  Convergence for both the modified Newton recovery scheme and CDA-Picard+CDA-Newton is rapid for each true solution and each initial guess.

\subsection{Comparison to ADAM}

\begin{table}[h!]
\centering
\begin{tabular}{|c|c|c|}
\hline
Iteration & $\nu_k$ &   \# CDA-P + CDA-N iters  \\ \hline
0 (initial guess) & 2.50000e-4  & - \\ \hline
1 & 1.90062e-4 &  5 \\ \hline
2 & 1.41042e-4 &  4 \\ \hline
3 & 1.10993e-4 &  4 \\ \hline
4 & 1.01252e-4 &  4 \\ \hline
5 & 9.99956e-5 &  4 \\ \hline
6 & 1.00000e-4 &  3 \\ \hline
\end{tabular}
\caption{\label{adam1} Shown above are convergence results for the ADAM optimizer for the 2D driven cavity problem with true $Re=10000$, $\nu_0=\frac{1}{4000}$, and $\alpha_{adam} = 6e-5$,
$\beta_1 = 0$. }
\end{table}

\begin{table}[h!]
\centering
\begin{tabular}{|c|c|c|c|}
\hline
$\alpha_{adam}$ & $\beta_1$ & ADAM iterations  \\ \hline
1e-5 & 0.5 & 41 \\ \hline
1e-5 & 0.25 & 33 \\ \hline
5e-5 &  0.5  & 31 \\ \hline
 5e-5 & 0.25 & 19 \\ \hline
 5e-5 & 0.1 &  11 \\ \hline
 5e-5 & 0.05 & 11 \\ \hline
 5e-5 & 0 & 9 \\ \hline
6e-5 & 0 & 6 \\ \hline
7e-5 & 0 & $>$100 \\ \hline
\end{tabular}
\caption{\label{adam2} Shown above is the number of ADAM iterations needed for convergence, for varying $\alpha_{adam}$ and $\beta_1$. }
\end{table}

While we propose viscosity recovery through a modified Newton method to find the minimizer/root, it is becoming quite popular to use the ADAM optimizer \cite{kingma2015adam}
for optimization problems.  Hence for comparison, we now repeat the 2D cavity test with true $\nu=\frac{1}{10000}$ and initial guess $\nu_0=\frac{1}{4000}$ but compare against using the computed sensitivity inserted into the ADAM optimizer to identify the $\hat{\nu}$ which minimizes the observable error.

The ADAM optimizer is a generalization of gradient descent which takes advantage of first order gradient information to achieve a more robust descent to the desired minima.  ADAM makes two primary adjustments beyond gradient descent.  The first is to use momentum to escape local minima.  The second adjustment enforces an adaptive learning rate so that coordinates (in multiple dimensional optimization) are weighted unevenly depending on the historic (through iterative time) size of the relative gradient.

Results are shown in Table \ref{adam1} for parameters $\alpha_{adam} = 6e-5$ and
$\beta_1 = 0$, and we observe that ADAM converges in 6 iterations.  Overall the convergence is superlinear, and although it takes 3 more iterations than modified Newton, the convergence is smoother (modified Newton overshoots from $\frac{1}{4000}$ to $\frac{1}{12973}$ in its first iteration) which in some instances may be a useful property.  However, this convergence with ADAM is highly parameter specific.  In Table \ref{adam2}, we show that with small changes to $\alpha_{adam}$ and $\beta_1$, significantly worse convergence is found.

One reason that ADAM may perform worse on this particular problem is that it was developed for, and works best for, stochastic gradient descent in very high dimensions and the current example, while acting on a high dimensional functional, is in reality a one-dimensional optimization/root-finding problem.  In addition, the Newton root-finding approach advocated here is very specific to this specific investigation of identifying the `correct' viscosity.  Such an algorithm is only viable when a `correct' or `true' solution exists.  If there is no `correct' solution such as in the case of an eddy viscosity for under-resolved simulations, a root-finding algorithm will likely perform more poorly than a minimization routine.  The quadratic nature of the loss is also very conducive to a Newton type approach, as Newton's method is typically derived using a locally quadratic assumption.

\subsection{The case of noisy solution data}

\begin{table}[h!]
\centering
\begin{tabular}{|c||c||c||c|}
	\hline
 &	$\gamma=0.01$ & $\gamma=0.001$ & $\gamma=0.0001$ \\ \hline
Iteration &	$\nu_k$ &   $\nu_k$  &  $\nu_k$   \\ \hline
0 &3.33333e-4  & 3.33333e-4  & 3.33333e-4  \\ \hline
1 & 1.90787e-4 & 1.90067e-4     & 1.90036e-4 \\ \hline
2 & 2.04710e-4 & 2.00027e-4      &  1.99947e-4 \\ \hline
3 & 1.90312e-4 & 2.16735e-4    &  2.00069e-4 \\ \hline
4 & 2.04532e-4 & 1.99849e-4    & 1.99947e-4 \\ \hline
5 & 1.89779e-4 & 2.02055e-4    &  2.00069e-4 \\ \hline
6 & 2.04351e-4 & 1.99846e-4     & 1.99948e-4 \\ \hline
7 & 1.89176e-4 & 2.02021e-4    &  2.00070e-4 \\ \hline
8 & 2.04162e-4  & 1.99843e-4  & 1.99948e-4 \\ \hline
9 & 1.88480e-4 & 2.02021e-4     & 2.00069e-4 \\ \hline
10 &  2.03967e-4  & 1.99839e-4   &1.99948e-4    \\ \hline
\hline
range x $10^{-4}$ & [1.884,2.040] & [1.998,2.020] & [1.999,2.001] \\ \hline 
\end{tabular}\\ \ \\
\caption{\label{noisetable} Shown above are convergence results for Algorithm \ref{modnewt}, using a true solution for $Re$=5000 with varying amounts of noise in the partial solution data for CDA.}
\end{table}

Our final numerical test is concerned with recovering viscosity in the setting where the partial solution data is inaccurate due to noise arising from e.g. measurement error or decompression of compressed data used to store or pass the solution.  Hence we now repeat the 2D cavity test above with $Re$=5000, but we add random noise to the partial solution data.  That is, we add uniformly distributed random numbers to the known true velocity solution via
\[
[I_H u^1,I_H u^2](x_j) = [I_H(u^1_{true}),I_H(u^2_{true}) ](x_j) + \gamma * \left( \mbox{unif}(-1,1); \mbox{unif}(-1,1) \right). 
\]

Results are shown in Table \ref{noisetable}, for varying $\gamma$=0.01, 0.001 and 0.0001.  We observe fast convergence of viscosity in all cases, but only up to a range that is linearly dependent on $\gamma$.  The width of the ranges are 1.56e-5 for $\gamma=10^{-2}$, 2.17e-6 for $\gamma=10^{-3}$ and 1.21e-7 for $\gamma=10^{-4}$.  Roughly, we observe that with $\mathcal O(10^{-n})$ magnitude noise, we obtain $n$ digits of accuracy within 4 or 5 iterations of Algorithm \ref{modnewt}.

\section{Conclusions and Future Directions}
\label{sec: Con}

An effective and very efficient method has been constructed, analyzed and tested for recovering an unknown viscosity $\nu$ for the steady NSE when partial solution data is available.  The method is a combination of two new algorithms proposed in this paper: a modified Newton iteration for recoving an unknown viscosity, and the CDA-Picard+CDA-Newton solver that the modified Newton iteration uses at each step.  The modified Newton iteration was constructed based on analysis that found the objective function derivative can be found with one additional linear solve and analysis that showed the objective function has a repeated root at $\nu$.  The CDA-Picard+CDA-Newton method was proven to converge quadratically, and its basin of convergence is proven to grow as more partial solution data is available.

Several directions for future work arise from our results herein.  The first is for noisy data.  While we did not analyze this case, we performed numerical tests that showed the proposed method still appears to work well, but only converges to within a noise dependent range of the true viscosity.  Proving such a result would be both important and useful.  Another important direction is multiple parameter recovery, which could be for the NSE when both the viscosity and forcing are unknown or in multiphysics problems such as Boussinesq or MHD that have multiple physical parameters.

\section{Acknowledgements}
 JR would like to thank and acknowledge Society for Industrial and Applied Mathematics and the SIAM Postdoctoral Support Program for their financial support.
 
\section{Data Availability}

Data will be made available on request.

\section{Declaration of competing interest}

The authors declare the following financial interests / personal relationships which may be considered as competing interests: Author LR reports financial support from the Department of Energy, grant DE-SC0025292. JPW was partially supported by NSF grant DMS-2510495 and CCF-2343286.

\bibliographystyle{abbrv}
\bibliography{cas-refs,references}

\end{document}